\documentclass[letterpaper]{amsart}
\usepackage{amsmath}
\usepackage{amssymb,latexsym}
\usepackage{amscd}
\usepackage{xy} \xyoption{all} \CompileMatrices
\usepackage[
    bookmarksnumbered = true,
    pdftitle          = {Transfer maps and nonexistence of joint determinant},
    pdfsubject        = {{joint determinant}},
    pdfkeywords       = {{Milnor K-theory, transfer maps, Goodwillie group}},
    pdfauthor         = {{Sung Myung}},
    ]{hyperref}

\usepackage{amsthm}
    \newtheorem{theorem}    {Theorem}       [section]

    \newtheorem{definition} [theorem]       {Definition}

\newcommand{\Z}{{\mathbb Z}}

\newcommand{\R}{{\mathbb R}}
\newcommand{\C}{{\mathbb C}}

\newcommand \M {{\mathcal M}}

\newcommand \Spec {{\operatorname{Spec }}}

\makeindex

\begin{document}

        \title{Triviality of a trace on the space of commuting trace-class self-adjoint operators}
        \author{Sung Myung}
        \address{Department of Mathematics Education, Inha University,
                  253 Yonghyun-dong, Nam-gu,
                    Incheon, 402-751 Korea }
        \email{s-myung1\char`\@inha.ac.kr}


\begin{abstract}
In the present article, we investigate a possibility of a real-valued map on the space of tuples of commuting trace-class self-adjoint operators,
which behaves like the usual trace map on the space of trace-class linear operators.
It turns out that such maps are related with continuous group homomorphisms from the Milnor's $K$-group of the real numbers into
the additive group of real numbers. Using this connection, it is shown that any such trace map must be trivial.
\end{abstract}

        \date{}
        \maketitle

        \section{Introduction} \label{intro}

Let $H$ be a Hilbert space whose elements represent physical states.
In quantum mechanics, observables are represented by self-adjoint linear operators on $H$ and they are simultaneously measurable without deviation
if the corresponding operators commute. This practice gives us a reason to study commuting self-adjoint operators on a Hilbert space.

A trace-class operator on a Hilbert space is a bounded linear operator for which a trace is well-defined (\cite{MR1468230}). The space of
such operators are important also because it enjoys a duality with the space of bounded linear operators because, for a given bounded linear
operator $T$ on $H$, we obtain a continuous linear functional on the space of trace-class operators defined by $A \mapsto Tr (AT)$.
For such a class of operators, the trace map is one of the most important invariants we may think of and so we may ask if it can be generalized
to tuples of commuting operators.

Let $X_l$ be the collection of $l$-tuples of commuting self-adjoint trace-class linear operators on some Hilbert space $H$.
Each element of $X_l$ can be represented by a symbol $(A_1, \dots, A_l, H)$, where $A_1, \dots, A_l$ are commuting self-adjoint trace-class
operators on some Hilbert space $H$. It becomes a monoid if we define the addition by the rule
$(A_1, \dots, A_l, H_1) + (B_1, \dots, B_l, H_2) = (A_1 \oplus B_1, \dots, A_l \oplus B_l, H_1 \oplus H_2)$.
There exists a natural projection map from $X_l$ to the set of Hilbert spaces and,
for each fixed $H$, the set $(X_l)_H$ of $l$-tuples of commuting trace-class self-adjoint operators on $H$, i.e,
the `fibre over $H$' can be equipped with a structure of a real Banach space, whose topology is used in this article even though the component-wise
addition will not be considered.
$X_l$ will be called the space of commuting trace-class self-adjoint operators in the present article.

When $l=1$, we have the usual trace map $Tr: X_1 \rightarrow \R$. It is well-known that the trace of a trace-class operator is equal to the
sum of its eigenvalues repeated according to their multiplicities (\cite{MR0370246}).
We ask whether there exists a `trace' on $X_l$ for $l \ge 2$.
\begin{definition} \label{trace}
A monoid homomorphism $\psi$ from $X_l$ into the additive group $\R$ of real numbers is said to be a trace map
if the following three conditions are satisfied:\\
$(i)$ For each fixed Hilbert space $H$, the restriction of $\psi$ on each fibre $(X_l)_H$ is continuous.\\
$(ii)$ $\psi$ is additive, that is, for arbitrary commuting self-adjoint trace-class operators $A_1, \dots, A_l, B$ on a Hilbert space $H$,
we have
\begin{multline*}
\psi (A_1, \dots, A_{i-1}, A_i+B, A_{i+1}, \dots, A_l, H) \\
= \psi (A_1, \dots, A_{i-1}, A_i, A_{i+1}, \dots, A_l, H) + \psi (A_1, \dots, A_{i-1}, B, A_{i+1}, \dots, A_l, H)
\end{multline*}
$(iii)$ If $h: H_1 \rightarrow H_2$ is an isomorphism of Hilbert spaces, then
$\psi (A_1, \dots, A_l, H_1) = \psi (h A_1 h^{-1}, \dots, h A_l h^{-1}, H_2)$. \\
$(iv)$ If we have commuting $A_1(t), \dots, A_l(t) \in GL_n(\R[t])$ such that $A_i(0)$ and $A_i(1)$ are
positive-definite symmetric matrices, then
we have $$\psi (\log A_1(0), \dots, \log A_l(0), \C^n) = \psi (\log A_1(1), \dots, \log A_l(1), \C^n).$$
\end{definition}

The reason behind the condition $(iv)$ is that a trace should be realized from a determinant through the exponential function and a
determinant of commuting invertible operators should be invariant for a polynomial homotopy, i.e., an $l$-tuple of commuting invertible matrices with
polynomial entries. The logarithm of a positive-definite symmetric matrix $A$ is given by $S(\log D) S^{-1}$, where $D$ is diagonal matrix with
positive real eigenvalues, $S$ is the orthogonal matrix of corresponding eigenvectors
and $\log$ is the natural extension of the usual real-valued logarithm.

We note that the conditions $(iii)$ and $(iv)$ are automatically satisfied if we require the following much stronger condition:

$(iii')$ We have $\psi (A_1, \dots, A_l, H_1) = \psi (B_1, \dots, B_l, H_2)$
whenever  $Tr(A_i)=Tr(B_i)$ and $H_1 \cong H_2$ for every $i=1, \dots, l$.

If $(iii')$ holds, then since the determinant of $A_i(0)$ is equal to the determinant of $A_i(1)$ for every $i$,
the trace of $\log A_i(0)$ is equal to the trace of $\log A_i(1)$ for every $i$ and thus we must have
$$\psi(\log A_1(0), \dots, \log A_l(0), H_0) = \psi(\log A_1(1), \dots, \log A_l(1), H_1).$$

Under this definition, we show that any trace map $\psi$ on $X_l$ must be identically 0 when $l \ge 2$.
The result may seem rather discouraging, but we hope that the use of Milnor's $K$-theory and related tools
may be useful in some constructions in the theory of linear operators.

        \section{Milnor's $K$-theory and the Goodwillie groups} \label{Milnor-theory}
We define the Milnor's $K$-group $K^M_l(k)$ of a field $k$ as follows (\cite{MR0349811}):
\begin{definition} \label{Milnor-symbol}
The $l$-th Milnor's $K$-group $K^M_l(k)$ is the additive quotient group of the tensor product
$(k^\times) ^{\otimes l}=k^\times \otimes k^\times \otimes \dots \otimes k^\times$ ($l$-times) by the
subgroup generated by the elements of the form $a_1 \otimes \dots \otimes a_i \otimes \dots \otimes a_j \otimes \dots \otimes a_l \in (k^\times) ^{\otimes l}$
where $a_i + a_j =1$ for some $1 \le i < j \le l$.
We denote by $\{a_1, a_2, \dots, a_l \}$, called a Milnor symbol,
the image of $a_1 \otimes a_2 \otimes \dots \otimes a_l \in (k^\times) ^{\otimes l}$ in $K^M_l(k)$.
\end{definition}
In particular, $K^M_1(k)$ is isomorphic to the multiplicative group $k^\times$ of nonzero elements,
but the group operation is written additively, e.g.,  $\{a\} + \{b\} = \{ab\}$.
Also, we have $\{ab, c_2, \dots, c_l\} = \{a,c_2, \dots, c_l\}+\{b,c_2, \dots, c_l\}$ in $K^M_l(k)$ by the definition.

On the other hand, for a finite dimensional $k$-vector space $V$, let $GL(V)$ be the general linear group which consists of invertible $k$-linear operators on $V$.
Then, the Goodwillie group $GW_l(k)$ of $k$ is defined as follows (\cite{MSMulti}):

\begin{definition} \label{GW_l}
$GW_l(k)$ is the abelian group generated by $l$-tuples of commuting operators $(A_1, \dots, A_l)$ ($A_1, \dots, A_l \in GL(V)$ for various
finite dimensional $k$-vector space $V$,
subject to the following four kinds of relations. \\
(i) (Identity Operators) $(A_1, \dots, A_l) = 0$ when $A_i$ for some $i$ is equal to the identity operator $I \in GL(V)$. \\
(ii) (Similar Operators) $(A_1, \dots, A_l) = (SA_1S^{-1}, \dots, SA_lS^{-1})$
for any isomorphism $S: V {\overset{\sim}{\rightarrow}} W$ of $k$-vector spaces. \\
(iii)  (Direct Sum) When $V$ and $W$ are finite dimensional $k$-vector spaces,
$$ (A_1, \dots, A_l) + (B_1, \dots, B_l) = \left( A_1 \oplus B_1, \dots, A_l \oplus B_l \right)$$
for commuting $A_1, \dots, A_l \in GL(V)$ and commuting $B_1, \dots, B_l \in GL(W)$. \\
(iv) (Polynomial Homotopy) $\displaystyle \left(A_1(0), \dots, A_l(0) \right) = \left(A_1(1), \dots, A_l(1)\right)$
for commuting automorphisms  $A_1(t), \dots, A_l(t)$ of a free module of finite rank over in $k[t]$,
where $k[t]$ is the polynomial ring over $k$ with the indeterminate $t$. \\
\end{definition}

In the above definition, note that we may choose only one vector space $V$ from each class of isomorphic vector spaces to resolve
any possible set-theoretic problem. A similar remark could have been also applied during the construction of our monoid $X_l$.
It is shown in \cite{WM} that $GW_l(k)$ is isomorphic to the motivic cohomology group $H^l_{\M} \bigl(\Spec \, k , \Z(l) \bigr)$.
It is one of the most important theorems in algebraic $K$-theory that $K^M_l(k)$ is isomorphic to $H^l_{\M} \bigl(\Spec \, k , \Z(l) \bigr)$
(\cite{MR992981}). More precisely, in \cite{MSMulti},
it is shown that the assignment which sends a Milnor symbol $\{a_1, \dots, a_l\}$ to $(a_1, \dots, a_l)$, where $a_i$ is
considered simply as $1 \times 1$-matrix in each coordinate, gives rise to an isomorphism $K^M_l(k) {\overset{\sim}{\rightarrow}} GW_l(k)$.

For $k = \R$, the field of real numbers,
it is known that $K^M_l(\R)$ is isomorphic to $\Z/2 \oplus U$, where $\Z/2$ is generated by the Milnor symbol $\{-1, \dots, -1\}$
and $U$ is a uniquely divisible abelian group (\cite{MR0260844}). A Milnor symbol $\{a_1, \dots, a_l\}$ belongs to $U$ if at least
one of $a_1, \dots, a_l$ is positive. In such a case, the element can be written as $\{b_1, \dots, b_l \}$ where all $b_1, \dots, b_l$
are positive.
Accordingly, we may write $GW_l(\R) = U' \cup ( GW_l(\R)-U')$ as a set, where $U' \cong U$ is the subgroup of index 2 in $GW_l(\R)$,
which is generated by the $l$-tuples of commuting positive-definite symmetric operators.

        \section{Triviality of a trace map on the space of $l$-tuples of commuting self-adjoint linear operators} \label{triviality}

Let $X_l$ be the space of $l$-tuples of commuting self-adjoint trace-class linear operators on some Hilbert space $H$
as defined in Section \ref{intro}.

\begin{theorem}
For $l \ge 2$, there does not exist a nontrivial trace map from the space $X_l$
into $\R$ as in Definition \ref{trace}.
\end{theorem}
\begin{proof}
We define $Y_l$ be the submonoid of $X_l$ consisting of the
symbols $(A_1, \dots, A_l, H)$ where $H$ is finite dimensional. Then $Y_l$ is dense in $X_l$ and
any trace map $\psi$ on $X_l$, when restricted to $Y_l$, induces a monoid homomorphism on $Y_l$ which, denoted
by $\psi |_{Y_l}$, enjoys the properties $(i)-(iv)$ of Definition \ref{trace}, in case of finite dimensional $H$'s.

Now let $P_l$ be the monoid whose elements are the symbols $(A_1, \dots, A_l, V)$ where $A_1, \dots, A_l$ are commuting
positive-definite symmetric linear operators on a finite dimensional real vector space $V$, where the addition is given by
$$(A_1, \dots, A_l, V_1) + (B_2, \dots, B_l, V_2) = (A_1 \oplus B_1, \dots, A_l \oplus B_l, V_1 \oplus V_2)$$
whenever it makes sense.

Then there is an obvious monoid homomorphism $P_l \rightarrow U'$ whose image generates $U'$ as a group.
In fact, this map is even surjective (\cite{MSMulti}).

We define the map $\log: P_l \rightarrow Y_l$ by $\log (A_1, \dots, A_l, V) = (\log A_1, \dots, \log A_l, V_\C)$,
where $\log$ is chosen to take values in real symmetric matrices.

$$\xymatrix@C=7pc@R=5pc{    Y_l \ar[dr]^{\psi |_{Y_l}}&    P_l \ar[l]_{\log} \ar@{->>}[r] \ar@{}[d]|{\circlearrowleft}  & U' \ar[dl]_{0}  \\
                         &  \R  &     } $$

If $\psi |_{Y_l}: Y_l \rightarrow \R$ is the restriction of a given trace map, then the composite $\psi |_{Y_l} \circ \log: P_l \rightarrow \R$
induces a continuous homomorphism from the the subgroup $U'$ of $GW_l(\R)$ into the additive group $\R$ of real numbers
since the four relations $(i)$-$(iv)$ in Definition \ref{GW_l}
are carried to a subset of the relations in Definition \ref{trace} under the map $\log: P_l \rightarrow Y_l$.

But, Theorem A.1. of \cite{MR0349811} shows that any continuous homomorphism from $U' \ \cong U \subset K^M_2(\R)$ into $\R$
should be trivial when $l=2$. The argument there can be easily generalized for $l \ge 3$ and accordingly we conclude that the induced map $U' \rightarrow Y_l$ is trivial.
Thanks to the existence of a simultaneous spectral resolution, for any $(A_1, \dots, A_l, H) \in Y_l$,
there exists some $(B_1, \dots, B_l, \C^n)$ in the image of the log map $\log : P_l \rightarrow Y_l$ such that
the operators $(A_1, \dots, A_l)$ and $(B_1, \dots, B_l)$ have the same joint spectrums when counted with multiplicities.
Therefore, by Definition \ref{trace} $(iii)$, we conclude that $\psi |_{Y_l} = 0$ on $Y_l$. Consequently, $\psi=0$ on the whole $X_l$ since $Y_l$ is dense in $X_l$.
\end{proof}

\section*{Acknowledgments}
This work was supported by INHA University Research Grant (INHA-35039).

     \bibliographystyle{mrl}

\begin{thebibliography}{1}

\bibitem{MR0370246}
J.~A. Erdos.
\newblock On the trace of a trace class operator.
\newblock {\em Bull. London Math. Soc.}, 6:47--50, 1974.

\bibitem{MR1468229}
Richard~V. Kadison and John~R. Ringrose.
\newblock {\em Fundamentals of the theory of operator algebras. {V}ol. {I}},
  volume~15 of {\em Graduate Studies in Mathematics}.
\newblock American Mathematical Society, Providence, RI, 1997.
\newblock Elementary theory, Reprint of the 1983 original.

\bibitem{MR1468230}
Richard~V. Kadison and John~R. Ringrose.
\newblock {\em Fundamentals of the theory of operator algebras. {V}ol. {II}},
  volume~16 of {\em Graduate Studies in Mathematics}.
\newblock American Mathematical Society, Providence, RI, 1997.
\newblock Advanced theory, Corrected reprint of the 1986 original.

\bibitem{MR0260844}
John Milnor.
\newblock Algebraic {$K$}-theory and quadratic forms.
\newblock {\em Invent. Math.}, 9:318--344, 1969/1970.

\bibitem{MR0349811}
John Milnor.
\newblock {\em Introduction to algebraic {$K$}-theory}.
\newblock Princeton University Press, Princeton, N.J., 1971.
\newblock Annals of Mathematics Studies, No. 72.

\bibitem{MSMulti}
Sung Myung.
\newblock On multilinearity and skew-symmetry of certain symbols in motivic
  cohomology of fields.
\newblock {\em http://arxiv.org/abs/0801.1405}.
\newblock preprint.

\bibitem{MR992981}
Yu.~P. Nesterenko and A.~A. Suslin.
\newblock Homology of the general linear group over a local ring, and
  {M}ilnor's {$K$}-theory.
\newblock {\em Izv. Akad. Nauk SSSR Ser. Mat.}, 53(1):121--146, 1989.

\bibitem{WM}
Mark Walker.
\newblock {\em Motivic complexes and the ${K}$-theory of automorphisms}.
\newblock Thesis, University of Illinois, 1996.

\end{thebibliography}

     \renewcommand{\indexname}{Index}
     \printindex

\end{document}